\documentclass[11pt,a4paper]{article}

\usepackage[english]{babel}
\usepackage{amsmath}
\usepackage{amsfonts}
\usepackage{amssymb}
\usepackage{graphicx}
\usepackage[square,numbers]{natbib}

\newtheorem{theorem}{Theorem}
\newtheorem{corollary}{Corollary}
\newtheorem{proof}{Proof}

\title{A new amplitude-frequency formula for non-conservative oscillators}

\author{Andrés García\\
        GIMAP (Grupo de Investigación en Multifísica Aplicada)\\
		Universidad Tecnológica Nacional-FRBB\\
		11 de Abril 461, Bahia Blanca, Buenos Aires, Argentina\\
		\texttt{andresgarcia@frbb.utn.edu.ar}} 

\begin{document}

\maketitle

\begin{abstract}
This paper formalize the existence's proof of first-integrals for any second order ODE, allowing to discriminate periodic orbits. Up to the author's knowledge, such a powerful result is not available in the literature providing a tool to determine periodic orbits/limit cycles in the most general scenario.
\end{abstract}

{\bf Keywords:} Periodic orbit, Non-conservative oscillator, Second order ODE.

\section{Introduction}

In a recent work (\cite{Garcia_2019_arxiv}) the existence of first computable integrals for periodic orbits was suggested. Moreover,  the application to the case reported by Mickens in \cite{Mickens_2005} is the initial key to develop a general amplitude-frequency formula for non-conservative oscillators.

The main purpose of this communication is the main theorem's improvement in \cite{Garcia_2019_arxiv}, removing a condition along with the proof's simplification, presenting  a new amplitude-frequency's formula for non-conservative oscillators. 

Up to the author's knowledge, no other available formula for non-conservative oscillators can be found. Amplitude-frequency formulas are exclusively developed for conservative oscillators: $\ddot{x(t)}=f(x(t))$.

\section{Amplitude-Frequency formula}

\begin{theorem}
 A second order ODE : $\ddot{x}(t)=f(x(t),\dot{x}(t)),\quad f\in\Re \times \Re\rightarrow \Re$ possess a periodic orbit: $\{x(0)=A\in \Re^{+},x(T)=A,\dot{x}(0)=0\}$, if and only if there exists a function  $\phi(x)\in \mathbb{C}^{1}(\Re)$, such that: $\frac{d\phi(x)}{dx}=\frac{f(x,\phi(x)}{\phi(x)},\quad \phi(A)=0$
 	
\end{theorem} 
 
\begin{proof}

\textbf{Necessity:}

If there exists $\phi(x)\in \mathbb{C}^{1}(\Re)$, such that: $\dot{x}(t)=\phi(x(t))$, then: $\ddot{x}(t)=f(x(t),\phi(x(t)))$. This new ODE is in fact a conservative oscillator (see for instance \cite{Landau_1982}, pp.29-34).

\textbf{Sufficiency:}

Performing an asymptotic expansion for $f(x,\dot{x})$ using an arbitrary bounded function $\zeta(t)\in \mathbb{C}^{1}(\Re)$ (see for instance \cite{De_Brujin_2010}):

	\begin{eqnarray*}
		f(x,\dot{x})\sim f(x+\eta \cdot \zeta,
		x+\eta \cdot \dot{\zeta}) + \frac{\partial f(x,\dot{x})}
		{\partial x}\vert_{\{x+\eta\cdot \zeta,\dot{x}+\eta\cdot 
		\dot{\zeta}\}} \cdot \left[x-(x+\eta\cdot \zeta)\right]+\\
		+ \frac{\partial f(x,\dot{x})}
		{\partial \dot{x}}\vert_{\{\dot{x}+\eta\cdot \dot{\zeta},
		\dot{x}+\eta\cdot \dot{\zeta}\}} \cdot \left[\dot{x}-
		(\dot{x}+\eta\cdot \dot{\zeta})\right], \quad (\eta 
		\rightarrow 0)	
	\end{eqnarray*}

Integrating and taking into account the hypothesis of periodic orbits existence:

	\begin{eqnarray*}
		\int_{0}^{T} f(x,\dot{x}) \cdot dt=0\sim \int_{0}^{T} f(x+
		\eta \cdot \zeta,	x+\eta \cdot \dot{\zeta}) \cdot dt + 
		\int_{0}^{T} \frac{\partial f(x,\dot{x})}{\partial x}
		\vert_{\{x+\eta\cdot \zeta,\dot{x}+\eta\cdot \dot{\zeta}\}} 
		\cdot \\
		\cdot \left[x-(x+\eta\cdot \zeta)\right] \cdot dt+
		\int_{0}^{T} \frac{\partial f(x,\dot{x})}
		{\partial \dot{x}}\vert_{\{\dot{x}+\eta\cdot \dot{\zeta},
		\dot{x}+\eta\cdot \dot{\zeta}\}} \cdot \left[\dot{x}-
		(\dot{x}+\eta\cdot \dot{\zeta})\right] \cdot dt=0, \quad 
		(\eta \rightarrow 0)	
	\end{eqnarray*}

Equivalently:

	\begin{equation*}
		\int_{0}^{T} \eta \cdot \zeta(t) \cdot \left[\frac{d}{dt} 
		\left(\frac{\partial f(x,\dot{x})}{\partial \dot{x}}\right)+
		\frac{\partial  f(x,\dot{x})}{\partial x} \right]=0,\quad 
		(\eta\rightarrow 0, \forall \zeta(t)\in \mathbb{C}^{1}(\Re))	
	\end{equation*}

Applying the fundamental lemma of calculus of variations for a test function $\eta \cdot \zeta(t)$(see for instance \cite{Gelfand_1963}, pp. 9) :

	\begin{equation}\label{Condition extrema}
		\frac{d}{dt} \left(\frac{\partial f(x,\dot{x})}{\partial 
		\dot{x}}\right)+\frac{\partial  f(x,\dot{x})}{\partial x} 
		=0			
	\end{equation}

This equation is a necessary condition for the first variation of the following functional:

	\begin{eqnarray*}
		min_{x(t)\in\Re} \underbrace{\int_{0}^{t} f(x(\sigma,
		\dot{\sigma})\cdot d\sigma}_{J} \\
		\text{such that:}\\
		\begin{cases}
			\ddot{x}=f(x,\dot{x})\\
			x(0)=A
		\end{cases}
	\end{eqnarray*}

In summary, the first variation $\delta J=0$ for every periodic trajectory $\{x=\zeta_{1}(t,A,\dot{x}(0)),\dot{x}=\zeta_{2}(t,A,\dot{x}(0)) \}$of $\ddot{x}=f(x,\dot{x})$, let's say: $\frac{\partial}{\partial \dot{x}(0)} \int_{0}^{t} f(\zeta_{1}(t,A,\dot{x}(0),\zeta_{2}(t,A,\dot{x}(0))\cdot d\sigma=0,\quad \forall t \in [0,T]$, then:

	\begin{equation*}
		\frac{\partial}{\partial \dot{x}(0)} f(\zeta_{1}(t,A,
		\dot{x}(0),\zeta_{2}(t,A,\dot{x}(0))=0,\quad \forall 
		t\in [0,T]
	\end{equation*}

This conclusion shows that actually: $f(x,\dot{x})=f(x)$, in other words, a conservative oscillator with $\phi(x)\in \mathbb{C}^{1}(\Re)$, such that: $\dot{x}(t)=\phi(x(t))$. This completes the proof.

\end{proof}

Finally, specializing this result to oscillators: $\ddot{x}(t)=
f_{1}(x)\cdot f_{2}(\dot{x}(t))$:

\begin{corollary}
A second order ODE : $\ddot{x}(t)=f_{1}(x(t)) \cdot f_{2}(\dot{x}(t))$ possess periodic orbits: $\{x(0)=A\in \Re^{+},x(T)=A,\dot{x}(0)=0\}$, with an amplitude-period formula: 
	
	\begin{eqnarray*}
		\begin{cases}
			\int_{0}^{\phi(x)} \frac{\phi}{f_{2}(\phi)} \cdot d\phi=
			\int_{A}^{x} f_{1}(x) \cdot dx\\
			T=4 \cdot \int_{0}^{A} \frac{1}{\phi(x)}\cdot dx
		\end{cases}
	\end{eqnarray*}
	
\end{corollary}

\subsection{Mickens' oscillator}

Applying the corollary to Mickens' oscillator: $\ddot{x}=-x\cdot(1+\dot{x}^{s})$, then: $T=4 \cdot \int_{0}^{A} \frac{dx}{\sqrt{e^{A^{2}-x^{2}}-1}}$. Notice that this formula, is not more than the exact amplitude-frequency formula obtained in \cite{Mickens_2005}.

\section{Conclusions}

In this paper a novel first computable integral to reduce a second order non-conservative oscillator: $\dot{x}=f(x,\dot{x})$ to a conservative one: $\dot{x}=f(x,\phi(x))$ is proved.

Up to the author's knowledge, this result has never been obtained, with the possibility to count for periodic orbits computing a reduced first order ODE: $\dot{x}=\phi(x)$.

The specialization of the results in this paper to ODE's: $\ddot{x}(t)=f_{1}(x(t)) \cdot f_{2}(\dot{x}(t))$ provides an amplitude-period formula for non-conservative oscillators.

\section*{Acknowledgments}
This work is supported by Universidad Tecnológica Nacional.

\end{document}